\documentclass[11pt]{article}
\usepackage{amsmath, amssymb, amsthm}
\usepackage{enumitem}                                 
\usepackage[margin=1.0in]{geometry}
\usepackage[T1]{fontenc}

\AtEndDocument{\bigskip{\footnotesize%
\textsc{Department of Mathematics, the Pennsylvania State University, University
Park, PA 16802, USA} \par 
  \textit{E-mail address} \texttt{ aka5983@psu.edu
} \par
\textit{E-mail address} \texttt{ zhirenw@psu.edu
}
}}

\newtheorem{theorem}{Theorem}[section]
\newtheorem{Remark} [theorem]{Remark}

\newtheorem{Counter-example}[theorem]{Counter example}
\newtheorem{Claim}[theorem]{Claim}

\newtheorem{Lemma}[theorem]{Lemma}

\newtheorem{Definition}[theorem]{Definition}
\newtheorem{Corollary}[theorem]{Corollary}

\newtheorem*{theorem*}{Theorem}

\newcommand{\ignore}[1]{}

\usepackage{xcolor}
\usepackage[pagebackref]{hyperref}
\hypersetup{
   colorlinks,
    linkcolor={red!60!black},
    citecolor={blue!60!black},
    urlcolor={blue!90!black}
}

\DeclareMathOperator*{\sign}{sign}

\allowdisplaybreaks[1]

\title{Arbitrarily slow decay in the M\"{o}bius disjointness conjecture}

\author{Amir Algom and Zhiren Wang}
\date{}

\begin{document}
\maketitle
\begin{abstract}
Sarnak's M\"{o}bius disjointness conjecture  asserts that for any zero entropy dynamical system $(X,T)$,  $\frac{1}{N} \sum_{n=1} ^N f(T^n x) \mu (n)= o(1)$ for every  $f\in \mathcal{C}(X)$  and  every $x\in X$. We construct examples showing that this $o(1)$ can go to zero arbitrarily slowly. In fact, our methods yield a more general result, where in lieu of $\mu(n)$ one can put any bounded sequence such that the {C}es\`aro mean of the corresponding sequence of absolute values does not tend to zero.
\end{abstract}
\section{Introduction}
 A topological dynamical system is a pair $(X,T)$ where  $X$ is  compact metric space and $T\in \mathcal{C}(X)$. If the system $(X,T)$ has zero topological entropy, then Sarnak's M\"{o}bius disjointness conjecture \cite[Main Conjecture]{Sarnak2012conjecture} predicts that 
\begin{equation} \label{Eq Sarnak}
\frac{1}{N}\sum_{n=1} ^N \mu(n)f(T^n x) = o(1),\quad \text{ for every } f\in \mathcal{C}(X) \text{ and  every } x\in X.
\end{equation}
Many special cases of Sarnak's Conjecture have been established: A very partial list of examples consists of  \cite{Bourgain2013Sarnak, Houcein2014Lem, Host2018Fran, Green2012Tao}. We refer to the surveys of Ferenczi, Ku\l aga-Przymus, and Lema\'{n}czyk \cite{Fere2018Lem} and of Ku\l aga-Przymus and Lema\'{n}czyk \cite{Lem2021Kol} for  excellent expositions on the subject, and many more  references.

The goal of this paper is to study the rate of decay in Sarnak's conjecture. That is, to study the nature of the $o(1)$ as in \eqref{Eq Sarnak}. We will show that there are systems for which this $o(1)$  decays to zero arbitrarily slowly. Nevertheless, all the examples we construct to this end satisfy Sarnak's conjecture. Here is our main result:
\begin{theorem} \label{Main Coro}
For every decreasing and strictly positive sequence  $\tau(n)\rightarrow 0$ there is a  dynamical system $(X,T)$ with zero topological entropy that satisfies:
\begin{enumerate}
\item There exist $x\in X$ and $f\in \mathcal{C}(X)$ such that:
$$\limsup_{N\rightarrow \infty} \frac{ \frac{1}{N} \sum_{n=1} ^N f(T^n x)\mu(n) }{\tau(n)} >0.$$

\item The  system $(X,T)$ satisfies Sarnak's conjecture \eqref{Eq Sarnak}.
\end{enumerate}
\end{theorem}
Several remarks are in order. First, Sarnak \cite[the remark following Main Conjecture]{sarna2010kmobius} remarks that  rates are not required in the conjecture, and this is formally justified by  Theorem \ref{Main Coro}. Secondly, it is natural to ask if Theorem \ref{Main Coro} may be upgraded by finding a zero  entropy dynamical system $(X,T)$ and $f\in \mathcal{C}(X)$ such that for every rate function $\tau$ we can find $x\in X$ that satisfies part (1) of Theorem \ref{Main Coro}. Doing so is as hard as solving the full M\"{o}bius disjointness conjecture: Indeed, by \cite[Corollary 10]{Abdal2018Lem}, if the conjecture is true then for every zero entropy system $(X,T)$ and $f\in \mathcal{C}(X)$, \eqref{Eq Sarnak} holds uniformly in $x\in X$. This cannot hold concurrently with the aforementioned upgraded version of Theorem \ref{Main Coro}. In other words, Theorem \ref{Main Coro} is conjecturally optimal. Next, we remark that in many cases (possibly in all cases), it is known \cite{Tao2012Blog} that a sufficiently fast rate in Sarnak's conjecture implies that the system $(X,T)$ satisfies a prime number Theorem (PNT) in the sense discussed in \cite[Section 11.2]{Fere2018Lem}. Thus, recent examples  \cite{kanigowski2020prime, fkl2021} of zero entropy systems failing to satisfy a PNT can be viewed as evidence  towards Theorem \ref{Main Coro}. We also mention some recent interesting examples constructed by Lian and Shi \cite{Lian2021Shi} that, while not directly related to Theorem \ref{Main Coro}, are similar in spirit to our work. Finally, we remark that our construction was partially inspired by the recent work of Dolgopyat,   Dong,  Kanigowski,  and N{\'a}ndori  \cite{dolgopyat2020flexibility},  where they exhibit some new classes of zero entropy smooth systems that satisfy the Central Limit Theorem

We will derive Theorem \ref{Main Coro} from a more general statement. This is the following Theorem, which forms the main technical result of this paper:
\begin{theorem} \label{Main Theorem}
For every decreasing and strictly positive sequence  $\tau(n)\rightarrow 0$ there is a zero entropy  dynamical system $(X,T)$ and some $f\in \mathcal{C}(X)$ that satisfy:
\begin{enumerate}
\item Every sequence $|a_n|\leq 1$ with $\limsup_{N\rightarrow \infty} \frac{1}{N}\sum_{n=1} ^N |a_n| >0$ admits some  $x\in X$ such that
$$\limsup_{N \rightarrow \infty} \frac{ \frac{1}{N} \sum_{n=1} ^N f(T^n x)a_n }{\tau(n)} >0.$$

\item The  system $(X,T)$ satisfies Sarnak's conjecture \eqref{Eq Sarnak}.
\end{enumerate}

\end{theorem}
In fact, we will show that any sub-sequence $N_j$ such that 
\begin{equation} \label{Eq Nj and theta}
\lim_{j\rightarrow \infty}  \frac{1}{N_j}\sum_{n=1} ^{N_j} |a_n| =\theta >0
\end{equation}
admits a further subsequence $N_{j_k}$ such that for all $k$ large enough 
$$\frac{1}{N_{j_k}} \sum_{n=1} ^{N_{j_k}} f(T^n x) a(n) \geq \theta \cdot   \tau(N_{j_k}).$$
We emphasize that in Theorem \ref{Main Theorem} the system $(X,T)$ and the function $f\in \mathcal{C}(X)$  only depend on the rate function $\tau$, while the point $x\in X$ depends also on the sequence $a_n$.

The derivation of Theorem \ref{Main Coro} from Theorem \ref{Main Theorem} is straightforward:   It is well known that the M\"{o}bius function $\mu$ satisfies
$$ \lim_{N\rightarrow \infty} \frac{\sum_{n=1} ^N |\mu(n)|}{N} = \frac{6}{\pi^2}>0,$$
see e.g. \cite[Corollary 1.6]{Bateman2004Diamond}. Thus, Theorem \ref{Main Theorem} applied with $a_n = \mu(n)$ gives Theorem \ref{Main Coro}.

We end this introduction with a brief explanation of our construction. We consider subshifts of $\left(\lbrace -1,0,1 \rbrace^\mathbb{N} \times \lbrace -1,0,1 \rbrace^\mathbb{Z},\, T\right)$, where $T(y,z)=(\sigma y,\, \sigma^{y_1} z)$ and $\sigma$ is the left shift. Given a rate function $\tau$ we first construct a certain rapidly growing sequence $q_k\rightarrow \infty$. We then construct a subshift such that its base comes from concatenating words of length $q_{k+1}-q_k$, that have non-zero entries at distance at least $q_k$ from each other. Our space $X$ is  a product of $4$ spaces constructed this way, together with a finite set $\lbrace 0,1,2,3 \rbrace$. The function $f$ is taken to be 
$$f( (y^{(0)}, z^{(0)}), \, (y^{(1)}, z^{(1)}),\, (y^{(2)}, z^{(2)}),\, (y^{(3)}, z^{(3)}),\, i) = z_0 ^{(i)}.$$

Given $a_n$ as in Theorem \ref{Main Theorem} part (1), our construction of the point $x\in X$ relies on the following  observation: Assuming $a_n \in \mathbb{R}$ (see the beginning of Section \ref{Section correlations} on why this is allowed), let $\gamma_n := \text{sign}(a_n)$, and let $N_j$, $\theta$ be as in \eqref{Eq Nj and theta}. For every $q,M \gg 1$ one may show that
$$\max_{ c,d\in [0,q]\cap \mathbb{Z}} \left\lbrace   \frac{1}{qM} \sum_{b=c} ^{q-1+c} \sum_{n=1} ^{M} \gamma (qn+c) \cdot a(qn+b),\, \frac{-1}{qM} \sum_{b=d} ^{q-1+d} \sum_{n=1} ^{M} \gamma (qn+d+1) \cdot a(qn+b) \right \rbrace \geq \frac{\theta}{4}.$$
Here we pick $k=k(j)$ in some convenient way, $q=q_k$ and $M\approx \frac{N_j}{q_k}$.  We then construct our point $x$ via working in one of the subshifts in our space  - the exact choice depends on certain technical issues coming from the relation between $N_j$ and $q_k$. To set up $x$, we carefully concatenate  pieces of arithmetic progressions in $\gamma$ or $-\gamma$ in the fiber (using the equation above), with the base living in the corresponding shift space and behaving nicely along the observable $f$. This will allow us to find a subsequence of $N_j$ where the linear correlations as in Theorem \ref{Main Theorem} part (1) are well approximated by the average giving the $\max$ in the equation above. Thus, with some more work, we bound these correlations  from below by $\tau(N_j)\cdot \theta$.

Finally, to derive part (2) of Theorem \ref{Main Theorem}, we apply the Matom\"aki-Radziwi\l{}\l{} bound \cite{MR16} on averages of multiplicative functions along short intervals. To do this, we exploit some strong periodic behaviour that exists in the systems we construct.

\section{Proof of Theorem \ref{Main Theorem} Part (1)}

\subsection{Preliminaries} \label{Section pre}
Let $(X,T)$ be a dynamical system, where we recall that $X$ is a compact metric space and $T\in \mathcal{C}(X)$. We denote the metric on $X$ by $d_X$. Let us recall Bowen-Dinaburg definition of topological entropy (as in e.g. \cite{Walters1982ergodic}): For every $n\in \mathbb{N}$ we define a metric on $X$ via
$$d_n(x,y) = \max \lbrace d_X (T^i (x),\, T^i(y)):\, 0\leq i <n \rbrace.$$
A Bowen ball $B_n (x,\epsilon)$ of depth $n$ centred at $x\in X$ of radius $\epsilon>0$ is the corresponding (open) ball in the metric $d_n$,
$$B_n (x,\epsilon) = \lbrace y\in X:\, d_n(x,y)<\epsilon\rbrace.$$
For any set $E\subseteq X$, let $N(E,n,\epsilon)$ denote the minimal number of Bowen balls of depth $n$ and radius $\epsilon$ needed to cover $E$.  The topological entropy of $(X,T)$ is then defined as
$$h(T):= \lim_{\epsilon \rightarrow 0} \left( \limsup_{n\rightarrow \infty} \frac{\log N(X, n,\epsilon)}{n} \right).$$

Next, let $\sigma : \lbrace-1, 0 ,1\rbrace^\mathbb{Z} \rightarrow \lbrace-1, 0 ,1\rbrace^\mathbb{Z}$ denote the left shift. On $\lbrace-1, 0 ,1\rbrace^\mathbb{Z}$ and $\lbrace-1, 0 ,1\rbrace^\mathbb{N}$ we define the metric
$$ d(x,y) = 3^{- \min \lbrace |n|:\, x_n \neq y_n \rbrace}.$$
Also, for every $x\in \lbrace-1, 0 ,1\rbrace^\mathbb{N}$ and $k>l\in \mathbb{N}$ let $x|_l ^k \in \lbrace -1,0,1\rbrace^{k-l}$ be the word
$$x|_{l} ^k := (x_l,x_{l+1},....,x_k),$$
and we use similar notation in the space $\lbrace-1, 0 ,1\rbrace^\mathbb{Z}$ as well. 
Next, let
$$Z:=\lbrace-1, 0 ,1\rbrace^\mathbb{N} \times \lbrace-1, 0 ,1\rbrace^\mathbb{Z}$$
and endow $Z$ with the sup-metric on both its coordinates. Note that open balls in this metric are also closed, and thus for every $n\in \mathbb{N}$, $x\in X$ and $\epsilon>0$ the Bowen ball $B_n (x,\epsilon)$ is closed. Also, we denote by $\Pi_i$, $i=1,2$, the coordinate projections in $Z$.   Finally,   we define the skew-product $T:Z\rightarrow Z$ via
$$T(y,z) = (\sigma(y),\, \sigma^{y_1} (z)).$$
We say that $X\subseteq Z$ is a subshift if it is closed and $T$-invariant. 

We will require the following Lemma:
\begin{Lemma} \label{Lemma metric properties}
The  system $(Z,T)$ satisfies that for every $n\in \mathbb{N}$, $\epsilon>0$, and $x=(y,z)\in Z$,
\begin{enumerate}

\item We have
$$T^n (y,\,z) = \left( \sigma^n y,\, \sigma^{\sum_{i=1} ^n y_i} z \right).$$

\item Let $m=m(n,y) = \min \lbrace \min_{1\leq k\leq n} \sum_{i=1} ^k y_i,\, 0 \rbrace$ and $M :=M(n,y)= \max \lbrace \max_{1\leq k\leq n} \sum_{i=1} ^k y_i,\, 0 \rbrace$. Then for any $l\in \mathbb{N}$ the Bowen ball $d_n(x, 3^{-l})$ equals
$$ \left\lbrace (a,b)\in Z: \, a|_1 ^{l+n} = y|_1 ^{l+n},\, b|_{-l+m} ^{l+M} = z|_{-l+m} ^{l+M}  \right\rbrace.$$

\item For any set $E\subseteq Z$
$$N(E, n ,\epsilon) = N(\text{cl}(E), n ,\epsilon),$$
where  $\text{cl}(E)$ is the closure of the set $E$.

\end{enumerate}
\end{Lemma}
\begin{proof}
Part (1) follows immediately from the definition of the map $T$. Part (2)  follows from part (1). Finally, part (3) is an immediate consequence of the fact that in  $(Z,\, T)$ Bowen balls are closed.
\end{proof}
\subsection{Construction of some zero entropy systems} \label{Section construction}
Fix a sequence $\tau(n) \rightarrow 0$ as in Theorem \ref{Main Theorem}. We begin by  constructing a rapidly growing  sequence $q_k \rightarrow \infty$ (that depends on $\tau$) such that for every $k\in \mathbb{N}$  we have:
\begin{enumerate}
\item  $q_{k+1} > q_{k} ^4 + 3q_k$.

\item $\tau( \frac{q_{k+1}}{3}) < \frac{1}{16q_k}$.

\end{enumerate}
We now use $q_k$ to define four sequences: 
$$q_k ^{(0)}:=q_{2k},\, q_k ^{(1)}=q_{2k+1}, \, q_k ^{(2)} := q_k ^{(0)} -1,\, q_k ^{(3)} := q_k ^{(1)} -1.$$
Notice that property (1) above also holds for $q_k ^{(i)}$ for every $i\in \lbrace 0,1,2,3\rbrace$. In particular, 
$$\lim_{k\rightarrow \infty} \frac{q_{k+1} ^{(i)} }{q_k ^{(i)}} = \infty, \text{ for every } i\in \lbrace 0,1,2,3\rbrace.$$
Next,  for every $i\in \lbrace 0,1,2, 3\rbrace$ and every $k$ let
$$A_k ^{(i)} := \lbrace j\cdot q_k ^{(i)}: \, j\in \mathbb{Z},\, q_k ^{(i)} \leq j\cdot q_k ^{(i)} \leq q_{k+1} ^{(i)} \rbrace.$$
For every $i\in \lbrace 0,1,2,3 \rbrace$ and every $k\in \mathbb{N}$ we construct elements $s^{(i)} _k \in \lbrace -1,0,1\rbrace^\mathbb{N}$ such that:
\begin{enumerate}
\item $ s^{(i)} _k (n) = 0$ for every integer $n\notin A_k ^{(i)} $.

\item For every $j\cdot q_k ^{(i)} \in A_k ^{(i)}$,
$$ s^{(i)} _k (j\cdot q_k ^{(i)})=1 \text{ if } j \leq \left[\frac{q_{k+1} ^{(i)} }{3 q_k ^{(i)}}\right],$$
and
$$s^{(i)} _k (j\cdot q_k ^{(i)} )=-1 \text{ if } \left[\frac{q_{k+1} ^{(i)} }{3 q_k ^{(i)}}\right] < j\leq 2\left[\frac{q_{k+1} ^{(i)} }{3 q_k ^{(i)}}\right].$$
\end{enumerate}
Next, for every element $x\in \lbrace -1,0,1\rbrace^\mathbb{N}$ and $p\in \mathbb{N}_0$ we define $\sigma^{-p} x\in \lbrace -1,0,1\rbrace^\mathbb{N}$ as $\sigma^{-p} x = x$ if $p=0$, and otherwise
$$\left( \sigma^{-p} x \right)|_1 ^p = (0,...,0), \text{ and for all } n>p, \, \, \sigma^{-p} x (n) = x(n-p).$$  The following Lemma is an immediate consequence of our construction.
\begin{Lemma} \label{Lemma Tent strucutre}
For every $i\in \lbrace 0,1,2, 3\rbrace$, $k\in \mathbb{N}$,  and $p = 0,...,q_k ^{(i)}$   we have
$$\sum_{j \in  [q_k ^{(i)}, \, q_{k+1} ^{(i)}) \cap \mathbb{Z} } \left( \sigma^{-p}  s^{(i)} _k \right)  (j ) =0.$$
\end{Lemma}
\begin{proof}
This follows since by our construction
$$ \left| \left \lbrace j\cdot q_k ^{(i)} \in A_k ^{(i)}: \, s^{(i)} _k (j\cdot q_k ^{(i)})=1 \right \rbrace \right| = \left| \left \lbrace j\cdot q_k ^{(i)} \in A_k ^{(i)}: \, s^{(i)} _k (j\cdot q_k ^{(i)})=-1 \right \rbrace \right|.$$
\end{proof}

Next, for every $i \in \lbrace 0,1,2,3 \rbrace$ and $k \in \mathbb{N}$ define the truncations
$$R^{(i)} _k = \left \lbrace \left( \sigma^{-p} s^{(i)} _k \right)|_{ q_k ^{(i)} } ^{ q_{k+1} ^{(i)}-1}  :\, p = 0,...,q_k ^{(i)} \right \rbrace \subseteq \lbrace -1,0,1\rbrace^{ q_{k+1} ^{(i)} - q_k ^{(i)}}.$$
We now define the space $P^{(i)}$ of all infinite sequences that have, for every $k$, some  word from $R^{(i)} _k$ between their $q_k ^{(i)}$ and $q_{k+1} ^{(i)}-1$ digits. Formally,
$$P^{(i)}  = \lbrace x\in \lbrace -1, 0, 1 \rbrace^\mathbb{N}:\,  x|_{ q_k ^{(i)} } ^{ q_{k+1} ^{(i)}-1} \in R^{(i)} _k,\, \text{ and } x|_1 ^{q_1 ^{(i)} -1} =(0,...,0) \rbrace. $$
The following Lemma is an immediate consequence of Lemma \ref{Lemma Tent strucutre}:
\begin{Lemma} \label{Lemma long tent}
For every $i\in \lbrace 0,1,2,3\rbrace$, $k\in \mathbb{N}$, and  $y\in P^{(i)}$,
$$\sum_{j=1} ^{q_k ^{(i)} -1} y(j) = 0$$
\end{Lemma}

Finally, for every $i\in \lbrace 0 ,1,2,3 \rbrace$ we define the subshift of $(Z,T)$
$$X_i = \text{cl} \left( \bigcup_{n\in \mathbb{N}_0}  T^n \left( P^{(i)} \times \lbrace -1,0,1\rbrace^\mathbb{Z} \right)  \right) .$$

\begin{Claim} \label{Zero entropy for each factor}
For every $i\in \lbrace 0,1,2,3\rbrace$ we have $h(X_i,\,T)=0$. 
\end{Claim}
\begin{proof}
Fix $n,u\in \mathbb{N}$. We count how many Bowen balls of radius $\frac{1}{3^u}$ and depth $n$ are needed to cover $X_{i}$.  Recall that we denote this quantity by $N(X_i, n,\frac{1}{3^u})$. By Lemma \ref{Lemma metric properties} part (3), this is the same  number as 
$$N \left( \bigcup_{l\in \mathbb{N}_0} T^l \left( P^{(i)} \times \lbrace -1,0,1\rbrace^\mathbb{Z} \right) ,\, n,\ \frac{1}{3^u} \right).$$
So,  we work with the latter space (i.e. without taking the closure).  

Let $k=k(n+u,i)$ be such that 
\begin{equation} \label{Eq for n}
q_k ^{(i)} \leq n+u < q_{k+1} ^{(i)}.
\end{equation}
Our first observation is that we can write
$$ \bigcup_{l\in \mathbb{N}_0} T^l \left( P^{(i)} \times \lbrace -1,0,1\rbrace^\mathbb{Z} \right) = A_1 \bigcup A_2 \bigcup A_3.$$
To define the sets $A_i$ we first note that  every $x\in \bigcup_{l\in \mathbb{N}_0} T^l \left( P^{(i)} \times \lbrace -1,0,1\rbrace^\mathbb{Z} \right)$ admits some $l\in \mathbb{N}_0$ and $\tilde{x} \in P^{(i)} \times \lbrace -1,0,1\rbrace^\mathbb{Z}$ such that $x = T^l \tilde{x}$. We denote by $p=p(x)\in \mathbb{N}$ the unique integer such that $q_{k-1} ^{(i)} +l \in [q_p ^{(i)},\,  q_{p+1} ^{(i)})$. Note that $p\geq k-1$. Then
$$A_1 = \lbrace x: p(x)\geq k+1\rbrace,\,   A_2 = \lbrace x: p(x)= k\rbrace,\, A_3 = \lbrace x: p(x)= k-1\rbrace.$$

Thus, we bound the covering numbers for $A_1-A_3$ separately. Before doing so, we notice that for any $x\in A_j$ for $j=1,2,3$  there are at most $3^{ q_{k-1} ^{(i)}}$ possibilities for the first $q_{k-1} ^{(i)}$ digits of  $\Pi_1 (x)$. 

\begin{enumerate}
\item Covering $A_1$: For any $x\in A_1$ the word $\left( \Pi_1 x \right) |_{ q_{k-1} ^{(i)}} ^{n+u}$ always consists of zeros separated by $1$ or $-1$, and in this case the non-zero entries appear at distance at least $q_{k+1} ^{(i)} >n+u$ from each other. Since there can be only one non-zero entry, there are at most $2(n+u)$ options for the configuration of this word.  So, with the notations of  Lemma \ref{Lemma metric properties} part (2), we see that 
$$|m|,M \leq  q_{k-1} ^{(i)}+1.$$
Thus, taking into account also the first $q_{k-1} ^{(i)}$ digits, and via Lemma \ref{Lemma metric properties} part (2),  the number of Bowen balls we need here is at most
$$\left( 3^{ q_{k-1} ^{(i)}} \times  2(n+u) \right) \times \left( 3^{u+q_{k-1} ^{(i)}+1} \right)^2.$$ 
\item Covering $A_2$: The word $\left( \Pi_1 x \right) |_{ q_{k-1} ^{(i)}} ^{n+u}$ consists of zeros separated by $1$ or $-1$, and in this case the first non-zero entries appear at distance at least $q_{k} ^{(i)} \leq n+u$ from each other. We also know that the first non-zero digit needs to appear within the first $q_{k} ^{(i)}$ digits. Another factor that needs to be taken into consideration is the possibility that $[q_{k-1} ^{(i)}+l,\, n+u+l]$ intersects  $[q_{k+1} ^{(i)},\, \infty)$. So, with the notations of  Lemma \ref{Lemma metric properties} part (2), we see that 
$$|m|,M \leq  q_{k-1} ^{(i)}+\frac{ n+u }{ q_{k} ^{(i)} }+1.$$

Taking all these factor into account, the number of Bowen balls we need here is at most
$$ \left( 3^{ q_{k-1} ^{(i)}}\times q_{k} ^{(i)}  \times 2(n+u)    \right) \times \left( 3^{u+ q_{k-1} ^{(i)} + \frac{ n+u }{ q_{k} ^{(i)} }+1} \right)^2.$$

\item Covering $A_3$: The word $\left( \Pi_1 x \right) |_{ q_{k-1} ^{(i)}} ^{n+u}$ consists of zeros separated by $1$ or $-1$, and in this case the first non-zero entries appear at distance at least $q_{k-1} ^{(i)}$ from each other. We also know that the first non-zero digit needs to appear within the first $q_{k-1} ^{(i)}$ digits.  Another factor that needs to be taken into consideration is the possibility that $[q_{k-1} ^{(i)}+l,\, n+u+l]$ intersects  $[q_{k} ^{(i)},\, \infty)$. So, with the notations of  Lemma \ref{Lemma metric properties} part (2), we see that 
$$|m|,M \leq  q_{k-1} ^{(i)}+\frac{q_{k} ^{(i)} }{ q_{k-1} ^{(i)} }+ \frac{ n+u }{ q_{k} ^{(i)} } +1.$$

Taking all these factor into account, the number of Bowen balls we need here is at most
$$\left( 3^{ q_{k-1} ^{(i)}}  \times q_{k-1} ^{(i)} \times q_{k} ^{(i)} \times 2(n+u)   \right) \times \left( 3^{u+ q_{k-1} ^{(i)}+\frac{q_{k} ^{(i)} }{ q_{k-1} ^{(i)} }+ \frac{ n+u }{ q_{k} ^{(i)} }+1} \right)^2 .$$
\end{enumerate}
Thus, we see that 
$$ N(X_i, n,\frac{1}{3^u}) \leq 3\cdot  \max_{i=1,2,3} N(A_i,n,\frac{1}{3^u}) = 3\cdot  N(A_3,n,\frac{1}{3^u}),$$
which has been computed in point (3) above. So, making use of \eqref{Eq for n},
\begin{eqnarray*}
\frac{\log N(X_i, n, \frac{1}{3^u} )}{n} &\leq& \frac{\log 3 + \log \left( 3^{ q_{k-1} ^{(i)}} \cdot 2(n+u)  \cdot q_{k-1} ^{(i)} \cdot q_{k} ^{(i)}   \right) \cdot \left( 3^{u+ q_{k-1} ^{(i)}+\frac{q_{k} ^{(i)} }{ q_{k-1} ^{(i)} }+ \frac{ n+u }{ q_{k} ^{(i)} }+1} \right)^2 }{n} \\
&\leq& \frac{\log  6}{n} + \frac{q_{k-1} ^{(i)}\cdot \log 3}{n} + \frac{\log (n+u)}{n} + \frac{2\log q_{k} ^{(i)}}{n} \\
&+& \frac{\left(u+q_{k-1} ^{(i)}+\frac{q_{k} ^{(i)} }{ q_{k-1} ^{(i)} }+ \frac{ n+u}{ q_{k} ^{(i)} } +1 \right)  \log 9  }{n} \\
&\leq &C_1 \cdot \left( \frac{ \log q_{k} ^{(i)}}{n} + \frac{q_{k-1} ^{(i)}}{n}+\frac{\log (n+u)}{n}+ \frac{q_{k} ^{(i)} }{ q_{k-1} ^{(i)} \cdot (n+u)}\cdot \frac{n+u}{n} + \frac{ n+u }{ q_{k} ^{(i)} \cdot n }    \right) \\
&\leq &C_1 \cdot \frac{n+u}{n} \cdot  \left( \frac{ 2\log (n+u)}{n} + \frac{q_{k-1} ^{(i)}}{q_{k} ^{(i)}}+ \frac{1 }{ q_{k-1} ^{(i)} } + \frac{1 }{ q_{k} ^{(i)} }    \right). \\
\end{eqnarray*}
Here $C_1$ is a large constant that depends variously on $u$ and the other constants appearing in the second equation. We conclude that, fixing $u$,
\begin{equation*} 
\lim_{n\rightarrow \infty} \frac{\log N(X, n, \frac{1}{3^u} )}{n}  =0,
\end{equation*}
and the Claim is proved.
\end{proof}

\subsection{Finding correlations along arithmetic progressions} \label{Section correlations}
Let $a_n$ be a sequence as in part (1) of Theorem \ref{Main Theorem}, that is, such that $\limsup_{N\rightarrow \infty} \frac{1}{N}\sum_{n=1} ^N |a_n| >0$. By moving to either $\Re (a_n)$ or  $\Im (a_n)$, we may assume $a_n$ is a real valued sequence. We define a new sequence $\gamma_n \in \lbrace -1, 0, 1\rbrace$ via
$$\gamma_n :=\sign(a_n).$$
In particular,
$$\limsup_{N\rightarrow \infty}  \frac{1}{N}\sum_{n=1} ^N \gamma_n \cdot a_n = \limsup_{N\rightarrow \infty}  \frac{1}{N}\sum_{n=1} ^N |a_n| >0.$$
Let $\theta := \limsup \frac{1}{N}\sum_{n=1} ^N |a_n| >0$, and let $N_j$ be a sub-sequence such that
$$\lim_{j\rightarrow \infty} \frac{1}{N_j}\sum_{n=1} ^{N_j} |a_n|= \theta.$$
\begin{Definition} \label{Definition i'}
For every $j\in \mathbb{N}$ large enough we define $k'=k(j)\in \mathbb{N}$ and $i'=i(j)\in \lbrace 0, 1\rbrace$ as the unique integers such that:
$$\text{If } N_j \in [\frac{q_{k'} ^{(1)}}{3},\, \frac{q_{k'+1}^{(0)}}{3}) \text{ then } i' =0, \text{ and }$$
$$\text{If } N_j \in [\frac{q_{k'+1} ^{(0)}}{3},\, \frac{q_{k'+1}^{(1)}}{3}) \text{ then } i' =1.$$
We also define an integer 
$$M_{k'} ^{(i')} := \left[\frac{N_j}{q_{k'}^{(i')}}\right]$$ 
\end{Definition}
Note that by definition and the construction of the sequence $q_k$
\begin{equation} \label{Eq propeties of M}
\frac{\left( q_{k'} ^{(i')} \right)^3}{3}=\frac{\left( q_{k'} ^{(i')} \right)^4}{3q_{k'} ^{(i')}} < M_{k'} ^{(i')} \leq \frac{q_{k'+1} ^{(i')}}{3q_{k'} ^{(i')}} .
\end{equation}

Next, recall the definition of $Z$ from Section \ref{Section pre} and let $g:Z\rightarrow \lbrace -1,0,1\rbrace$ be the function
$$g(y,z) = z_0.$$
For every $q,M \gg 1$ and $r,c$ such that $r,c\in [0,q]$ let
$$A_{r,c} ^{q ,M}:=  \frac{1}{qM} \sum_{b=r} ^{q-1+r} \sum_{n=1} ^{M} \gamma (qn+c) \cdot a(qn+b).$$
Finally, we also define
$$M_{k'} ^{(i'+2)} := \left[ \frac{q_{k'} ^{(i')} M_{k'} ^{(i')} }{q_{k'} ^{(i')}-1} \right] = \left[ \frac{q_{k'} ^{(i')} M_{k'} ^{(i')}}{q_{k'} ^{(i'+2)}} \right]$$
and note that $M_{k'} ^{(i'+2)} \approx M_{k'} ^{(i')} $. In the following Lemma we use the construction from Section \ref{Section construction}.
\begin{Lemma} \label{Key Lemma}
For every $j$ and $u\in \lbrace 0,1\rbrace$, writing $\ell = i'+2u$,  for every two integers $c,r \in [0,  q_{k'} ^{(\ell)}]$ let $x \in P^{(\ell)} \times \lbrace -1,0,1\rbrace^\mathbb{Z} \subseteq X_\ell$ be any element such that for every  $ q_{k'} ^{(\ell)} \leq n< q_{k'+1} ^{(\ell)}$
$$x (n) = \left( s_{k'} ^{(\ell)}(n-r),\, \gamma( q_{k'} ^{(\ell)} \cdot n+c) \right).$$
Then 
$$\frac{1}{q_{k'} ^{(\ell)}M_{k'} ^{(\ell)}  } \sum_{n=1} ^{q_{k'} ^{(\ell)}M_{k'} ^{(\ell)}} g(T^n x ) a(n) = A_{r,c} ^{q_{k'} ^{(\ell)}, \, M_{k'} ^{(\ell)} } + O\left( \frac{ q_{k'} ^{(\ell)} }{M_{k'} ^{(\ell)}} \right).$$
\end{Lemma}
Note that by the construction of $P^{(\ell)} \times \lbrace -1,0,1\rbrace^\mathbb{Z}$ in Section \ref{Section construction}, there exists  an element $x$ as in the statement of the Lemma in that space.
\begin{proof}
In this proof we suppress the $\ell,k'$ in our notation and simply write $q,\, M.$ 
First, for every two integers $j\in [1,\,M]$ and $b\in [r,\, q+r-1]$,
\begin{eqnarray*}
\sum_{d=1} ^{qj+b} \left( \Pi_1  x   \right) (d) &=& \sum_{d=1} ^{q-1} \left( \Pi_1  x  \right) (d)+ \sum_{d=q} ^{qj+b-1} \left( \Pi_1  x  \right) (d)\\
&=& \sum_{d=q} ^{qj+b-1} s_k ^{(\ell)} (d-r)\\
&=& \sum_{d=q-r} ^{qj+b-r-1} s_k ^{(\ell)} (d) = j.\\
\end{eqnarray*}
Note the use of Lemma \ref{Lemma long tent} in the second equality, and the use of the definition of $s_k ^{(\ell)}$ together with the fact that $M\leq \frac{q_{k'+1} ^{(\ell)} }{3 q_{k'} ^{(\ell)}}$ in the last one. Therefore,
\begin{eqnarray*}
\frac{1}{q M  } \sum_{n=1} ^{q M} g(T^n x) a(n) &=& \frac{1}{q M  } \sum_{n=q} ^{q M} g(T^n x) a(n) + O\left( \frac{1}{M  } \right) \\
&=& \frac{1}{q M  } \sum_{j=1} ^{M} \sum_{b=r} ^{q+r -1} g(T^{q\cdot j+b} x) a(q\cdot j+b) + O\left( \frac{1}{ M  } \right) \\
&=& \frac{1}{q M  } \sum_{j=1} ^{M} \sum_{b=r} ^{q+r -1} g\left( \sigma^{qj+b} \Pi_1 x, \, \sigma^{ \sum_{d=1} ^{qj+b} \left( \Pi_1  x  \right) (d)}  \Pi_2 x \right)   a(q\cdot j+b)  \\
&+&  O\left( \frac{1}{ M  } \right)\\
&=& \frac{1}{q M  } \sum_{j=1} ^{M} \sum_{b=r} ^{q+r -1} g\left( \sigma^{qj+b-r} s_{k'} ^{(\ell)}, \, \sigma^{j} \Pi_2 x \right)   a(q\cdot j+b)  \\
&+&  O\left( \frac{1}{ M  } \right)\\
&=& \frac{1}{q M  } \sum_{j=1} ^{M} \sum_{b=r} ^{q +r-1} \gamma(q \cdot j+c) \cdot  a(q\cdot j+b) + O\left( \frac{ q }{ M  } \right) \\
&=& A_{r,c} ^{q ,M} + O\left( \frac{ q }{ M  } \right).\\
\end{eqnarray*}
Indeed: The first equality follows since $g(T^n x)$ and $a_n$ are both bounded sequences, in the third equality we use Lemma \ref{Lemma metric properties} part (1), and in the fourth equality we are using the previous equation array and the definition of $x$. This definition along with the definition of $s_k ^{(\ell)}$ justify the fifth equality. The last equality is simply the definition of $A_{r,c} ^{q ,M}$.
\end{proof}

\begin{Remark}  \label{Remark other possibility}
In the setup of Lemma \ref{Key Lemma}, we may similarly find another $x \in P^{(\ell)} \times \lbrace -1,0,1\rbrace^\mathbb{Z}$  that satisfies the conclusion  of Lemma \ref{Key Lemma}, but for  $-A_{r,c} ^{q_{k'} ^{(\ell)}, M_{k'} ^{(\ell)} }$. Indeed, this follows from the very same proof by picking $x \in P^{(\ell)} \times \lbrace -1,0,1\rbrace^\mathbb{Z}$ to be any element such that for every  $ q_{k'} ^{(\ell)} \leq n<q_{k'+1} ^{(\ell)}$
$$x (n) = \left( s_{k'} ^{(\ell)}(n-r),\, -\gamma( q_{k'} ^{(\ell)} \cdot n+c) \right).$$
\end{Remark}

We will also require the following Lemma:

\begin{Lemma} \label{Lemma intermediate}
For every $j$ large enough there  is either some $c\in [0 , q_{k'} ^{(i')} )$ such that
\begin{equation}\label{Eq max}
 A_{c,c}  ^{q_{k'} ^{(i')} ,M_{k'} ^{(i')} } \geq \frac{\theta}{8q_{k'} ^{(i')}} ,
\end{equation}
or some  $d\in [0 , q_{k'} ^{(i'+2)} )$ with
$$ -A_{d+1,d} ^{q_{k'} ^{(i'+2)}, M_{k'} ^{(i'+2)}} \geq \frac{ \theta }{8q_{k'} ^{(i')}}.$$
\end{Lemma}
\begin{proof}
In this proof we again suppress the $i',k',u$ in our notation, and write instead $q,\, M,$ for $q_{k'} ^{(i')}$ and $M_{k'} ^{(i')}$, respectively (the terms corresponding to $i'+2$ will come up in the proof later). Now, for every $c,r \in [0,  q]$,
$$\sum_{c=0} ^{q-1} A_{c+r,c} ^{q,M} = \frac{1}{qM} \sum_{m=1} ^{qM} \gamma(m)\cdot \left( a(m+r)+...+a(m+r+q-1) \right) + O\left( \frac{1}{M} \right)$$
So,
$$qM \cdot  \sum_{c=0} ^{q-1}  A_{c,c} ^{q,M} = \sum_{m=1} ^{qM} \gamma(m)\cdot \left( a(m)+...+a(m+q-1) \right)+ O\left( q \right)$$
and
$$ (q-1) \left[ \frac{q M}{q-1} \right] \sum_{c=1} ^{q-1} A_{c+1,c} ^{q-1, \left[ \frac{q M}{q-1} \right] } = \sum_{m=1} ^{qM} \gamma(m)\cdot \left( a(m+1)+...+a(m+q-1) \right)+  O\left( q^2 \right)$$

Combining the last two displayed equations, 
$$qM \cdot  \sum_{c=0} ^{q-1}  A_{c,c} ^{q,M}  -  (q-1) \left[ \frac{q M}{q-1} \right] \sum_{c=1} ^{q-1} A_{c+1,c} ^{q-1, \left[ \frac{q M}{q-1} \right] }  = \sum_{m=1} ^{qM} \gamma(m) a(m) +O\left( q^2 \right) \geq \theta/2 \cdot qM+O\left( q^2 \right).$$
It follows that, assuming $q$ is large enough, and via \eqref{Eq propeties of M}
$$\sum_{c=0} ^{q-1}  A_{c,c} ^{q,M} - \sum_{d=1} ^{q-1} A_{d+1,d} ^{q-1, \left[ \frac{qM}{q-1} \right] } \geq \theta/2- O\left( \frac{q}{M} \right) \geq \theta/2- O\left( \frac{1}{q^2} \right) \geq \theta/4.$$
Recalling our definition of $q_{k'} ^{(i'+2)}$ and $M_{k'} ^{(i'+2)}$, this implies the Lemma.
\end{proof}

\subsection{Construction of the point and system as in Theorem \ref{Main Theorem}} \label{Section main proof}
Recalling Lemma \ref{Lemma intermediate}, by perhaps moving to a further subseqeunce, we may assume that the  inequality from Lemma \ref{Lemma intermediate} is always given by the term corresponding to $q_{k'} ^{(i'+2u)}$ where $u=u(j)$ is either $0$ or $1$, and both $i'=i(j)$ and $u$ are assumed to be constant in $j$. Let us denote this constant value $i'+2u \in \lbrace 0,1,2,3\rbrace$ by $\ell$.  Recalling Definition \ref{Definition i'}, and passing to a subsequence if needed, we assume that the map $j\mapsto k(j)=k'$ is injective. 

We now construct a point $x^{(\ell)}\in P^{(\ell)} \times \lbrace -1,0,1\rbrace^\mathbb{Z} \subseteq X_\ell$ as follows: For every $j\in \mathbb{N}$ and $q_{k(j)} ^{(\ell)} \leq n < q_{k(j)+1} ^{(\ell)}$,  $x^{(\ell)} (n)=x(n)$  where $x$ is the element as in Lemma \ref{Key Lemma} (if $u= 0$) or Remark \ref{Remark other possibility} (if $u =1$), corresponding to $j$, $\ell$ as in the paragraph above, and either $r=c$ and $c$ (if $u=0$) or $r=d+1$ and $c=d$ (if $u=1$) yielding the inequality from Lemma \ref{Lemma intermediate}. Note that here we need the map $j\mapsto k(j)$ to be injective so this is well defined (i.e. the intervals $[q_{k(j)} ^{(\ell)}, \, q_{k(j)+1} ^{(\ell)})$ don't overlap). Note that so far we have only specified the digits $n \in \bigcup_{j\in \mathbb{N}} [q_{k(j)} ^{(\ell)}, \, q_{k(j)+1} ^{(\ell)})$,  and (since we have passed to a subsequence) it is possible that this union does  not cover  all of $\mathbb{N}$. So,  for all  digits not covered we make some choice that ensures $x^{(\ell)} \in P^{(\ell)} \times \lbrace -1,0,1\rbrace^\mathbb{Z}$. Note that by Lemma  \ref{Key Lemma}  and the construction of $P^{(\ell)}$, such a choice is readily available.

We now take our space to be 
\begin{equation}\label{EqExample} X:= X_0 \times X_1 \times X_2 \times X_3 \times \lbrace 0,1,2,3\rbrace,\end{equation}
with the self-mapping $\hat{T} \in \mathcal{C}(X)$ being
$$\hat{T}(p^{(0)},\, p^{(1)},\, p^{(2)},\, p^{(3)},\, i) =  (T p^{(0)},\, T p^{(1)},\, Tp^{(2)},\, Tp^{(3)},\, i).$$
The function $f\in \mathcal{C}(X)$ is taken to be 
$$f( (y^{(0)}, z^{(0)}), \, (y^{(1)}, z^{(1)}),\, (y^{(2)}, z^{(2)}),\, (y^{(3)}, z^{(3)}),\, i) = z_0 ^{(i)}.$$
We next choose our point $x$ to be any $x\in X$ such that: Its projection to $X_\ell$ is $x^{(\ell)}$, and its projection to $\lbrace0,1,2,3\rbrace$ is $\ell$.

We now prove part (1) of Theorem \ref{Main Theorem} via the following two claims:

\begin{Claim}
We have $h(X, \hat{T})=0$.
\end{Claim}
\begin{proof}
By Claim \ref{Zero entropy for each factor} each factor in the product space $X$ has zero entropy,  which implies the assertion via standard arguments. 
\end{proof}

\begin{Claim}
For all $j$ large enough,
$$ \frac{1}{N_j} \sum_{n=1} ^{N_j} f(\hat{T}^n x) a(n) \geq \theta\cdot \tau(N_j).$$
In particular,
$$\limsup_{N\rightarrow \infty} \frac{ \frac{1}{N} \sum_{n=1} ^{N} f(\hat{T}^n x) a(n)}{\tau(N)} >0.$$ 
\end{Claim}
\begin{proof}
Fix $j$ large, and let us write $N, q, M, x$, suppressing the dependence on $k',\ell,j$ (except in parts of the proof where we wish to emphasize this dependence). Note that
$$q M \in  [N-q, N].$$
Now:
\begin{eqnarray*}
\frac{1}{N} \sum_{n=1} ^{N} f(\hat{T}^n x) a(n) &=& \frac{1}{q M } \sum_{n=1} ^{q M} f(\hat{T}^n x) a(n)+ O( \frac{1}{  M} )\\
&=& \frac{1}{q M } \left( \sum_{n=1} ^{q -1} f(\hat{T}^n x) a(n) + \sum_{n=q} ^{q M } f(\hat{T}^n x) a(n) \right)+ O( \frac{1}{  M} )\\
&=& \frac{1}{q M  } \left( \sum_{n=1} ^{q -1} f(\hat{T}^n x) a(n) + \sum_{n=1} ^{q M } g(T^n x^{(\ell)} ) a(n) - \sum_{n=1} ^{q -1} g(T^n x^{(\ell)} ) a(n) \right)\\
&+&  O( \frac{1}{  M} )\\
&=&  \frac{1}{q_{k'} ^{(\ell)} M_{k'} ^{(\ell)} } \sum_{n=1} ^{q_{k'} ^{(\ell)} M_{k'} ^{(\ell)} } g(T^n x^{(\ell)}) a(n)  + O( \frac{ 1 }{M_{k'} ^{(\ell)}}  )\\
&\geq & \frac{\theta}{8q_{k'} ^{(i')}} + O\left( \frac{ q_{k'} ^{(\ell)} }{M_{k'} ^{(\ell)}}  \right)\\
\end{eqnarray*}
Note that in the third equality we are again using Lemma \ref{Lemma long tent} in a similar fashion to the proof of Lemma \ref{Key Lemma}, which is allowed since $x^{(\ell)}\in P^{(\ell)} \times \lbrace -1,0,1\rbrace^\mathbb{Z}$. For the last inequality we are using Lemmas \ref{Lemma intermediate} and \ref{Key Lemma} along with the definition of $x$.

We conclude that
$$\frac{1}{N_j} \sum_{n=1} ^{N_j} f(\hat{T}^n x) a(n) \geq \frac{\theta}{8q_{k'} ^{(i')}} + O\left( \frac{ q_{k'} ^{(\ell)} }{M_{k'} ^{(\ell)}}  \right).$$
By \eqref{Eq propeties of M}, 
$$O\left( \frac{ q_{k'} ^{(\ell)} }{M_{k'} ^{(\ell)}}  \right) \leq O\left( \left( \frac{1}{q_{k'} ^{(i')}} \right)^2 \right),$$
and so, as long as $j$ is large enough,
$$\frac{1}{N_j} \sum_{n=0} ^{N_j-1} f(\hat{T}^n x) a(n) \geq \frac{\theta}{16q_{k'} ^{(i')}}.$$
Finally, it follows from our choice of  $N_j$ that $N_j$ is larger than  the element of the sequence $q_k /3$ that comes after $q_{k'} ^{(i')} /3$. So, by the choice of the sequence $q_k$,
$$\frac{1}{16q_{k'} ^{(i')}} \geq \tau(N_j).$$
Combining the last two displayed equations implies the Claim.

\end{proof}

\section{Proof of Theorem \ref{Main Theorem} Part (2)}

In this Section we prove Part (2) of Theorem \ref{Main Theorem}. That is, we show that the system $(X, \hat{T})$ given in \eqref{EqExample} satisfies the M\"obius disjointness conjecture \eqref{Eq Sarnak}.  The proof will be an application of Matom\"aki-Radziwi\l{}\l{}'s bound \cite{MR16} on averages of multiplicative functions along short intervals, which has recently become a standard tool to establish M\"obius disjointness for systems with strong periodic behaviour.

Denote a point $x\in X$ as 
$$(x^{(0)},x^{(1)},x^{(2)},x^{(3)},i), \text{ where } x^{(\ell)}=(y^{(\ell)},z^{(\ell)}).$$

For each $p=(y,z)\in \{-1,0,1\}^{\mathbb N}\times \{-1,0,1\}^{\mathbb Z}$ and $M\in \mathbb N$, denote by $[p]_M$ the truncation 
$$[p]_M:=\big((y_1,\cdots,y_M),(z_{-M},\cdots,z_M)\big).$$
 Write $\mathcal C_{M} (X)$ for the space of cylinder functions $f(x)$ that only depends on $([x^{(\ell)}]_M)_{0\leq \ell\leq 3}$ and the fifth coordinate $i \in \lbrace 0,1,2,3 \rbrace$. Then $\bigcup_{M=1}^\infty\mathcal C_{M} (X)$ is dense in $\mathcal C(X)$ with respect to $C^0$ norm. In consequence, it suffices to verify \eqref{Eq Sarnak} for all cylinder functions $f \in \mathcal C_{M}(X)$ for every $M$. 

The main technical Lemma that we need is the following:
\begin{Lemma}\label{LemLongPeriod} For all $0\leq \ell\leq 3$ and $M, H\in\mathbb N$ and $x\in X$, there exists a set $\Lambda^{(\ell)}(M,H,x)\subseteq\mathbb N$ that satisfies:
\begin{enumerate}
\item $\lim_{N\to\infty}\frac 1N\#\big(\{1,\cdots, N\}\cap\Lambda^{(\ell)}(M,H,x)\big)=1$.

\item For all $n\in \Lambda^{(\ell)}(M,H,x)$,
$[T^{n+h}x^{(\ell)}]_M$ is constant for $0\leq h\leq H-1$.
\end{enumerate}
\end{Lemma}
 \begin{proof}
Since 
$$x^{(\ell)}\in X_{\ell}= \text{cl} \left(\bigcup_{b\in\mathbb N _0} T^b(P^{(\ell)}\times\{-1,0,1\}^\mathbb Z)\right),$$
for each $\ell$ and all $N\in\mathbb N_0$, there exists $x^{(N,\ell)}\in \bigcup_{b\in\mathbb N_0} T^b(P^{(\ell)}\times\{-1,0,1\}^{\mathbb Z})$ such that 
$$[x^{(N,\ell)}]_N=[x^{(\ell)}]_N \text{ for all } n\leq N.$$
We also choose $b^{(N,\ell)}\in\mathbb N_0$ and $\tilde x^{(N,\ell)}\in P^{(\ell)}\times\{-1,0,1\}^{\mathbb Z}$ such that $x^{(N,\ell)}=T^{b^{(N,\ell)}}\tilde x^{(N,\ell)}$.

Then for $1\leq n\leq N$ and $0\leq h\leq H-1$, $$[T^{n+h}x^{(\ell)}]_M=[T^{n+h}x^{(N+H+M,\ell)}]_M=[T^{n+b^{(N+H+M,\ell)}+h}\tilde x^{(N+H+M,\ell)}]_M.$$
Therefore, by part (1) of Lemma \ref{Lemma metric properties}, $[T^{n+h}x^{(\ell)}]_M$ is constant for $0\leq h\leq H-1$ if 
\begin{equation}\label{EqLongPeriod1}\Pi_1\tilde x^{(N+H+M,\ell)}(n+b^{(N,\ell)}+h')=0, \text{ for all } 0\leq h'\leq H+M-1.\end{equation}
Since $ \tilde x^{(N+H+M,\ell)}\in  P^{(\ell)}\times\{-1,0,1\}^{\mathbb Z}$, for every $k\in \mathbb{N}$ there is some $0\leq r_k^{(\ell)}\leq q_k ^{(\ell)}-1$ such that 
$$\Pi_1\tilde x^{(N+H+M,\ell)}(n')=s_k^{(\ell)}(n'- r_k^{(\ell)})  \text{ for } q_k^{(\ell)}\leq n'<q_{k+1}^{(\ell)}.$$ 
In particular, $\Pi_1\tilde x^{(N+H+M,\ell)}(n')=0$ for all $q_k^{(\ell)}\leq n'<q_{k+1}^{(\ell)}$ with $n'\not\equiv  r_k^{(\ell)} (\text{mod } q_k^{(\ell)})$. 

It follows that for each $k$, \eqref{EqLongPeriod1} holds on the set $$\begin{aligned}\Lambda_{N,k}^{(\ell)}(M,H,x):=&\{1\leq n\leq N: q_k^{(\ell)}\leq n+b^{(N+H+M,\ell)}\leq q_{k+1}^{(\ell)}-H-M;\\
&\quad  n+b^{(N+H+M,\ell)}\not\equiv r_k^{(\ell)}-H-M+1, \cdots, r_k^{(\ell)}-1, r_k^{(\ell)} (\text{mod } q_k^{(\ell)})\}.\end{aligned}$$
Set $\Lambda_N^{(\ell)}(M,H,x)=\bigcup_{k=1}^\infty\Lambda_{N,k} ^{(\ell)}(M,H,x)\subseteq\{1,\cdots, N\}$. Then $[T^{n+h}x^{(\ell)}]_M$ is constant for $0\leq h\leq H-1$ if $n\in\Lambda^{(\ell)}(M,H,x)$.

Finally, 
$$\lim_{N\to\infty}\frac 1N\#\big(\{1,\cdots, N\}\cap\Lambda^{(\ell)} _N (M,H,x)\big)=1$$
because of the following facts: $H$ and $M$ are fixed,  $b^{(N+H+M,\ell)}\geq 0$, $\lim_{k\to\infty}q_k^{(\ell)}=\infty$ and $\lim_{k\to\infty}\frac{q_{k+1}^{(\ell)}}{q_k^{(\ell)}}= \infty$. We conclude the proof by defining 
$$\Lambda^{(\ell)}(M,H,x):=\bigcup_{N=1}^\infty\Lambda_N^{(\ell)}(M,H,x).$$
\end{proof}

\begin{Corollary}\label{CorLongPeriod} For all $ M, H\in\mathbb N$ and $x\in X$, there exists a set $\Lambda(M,H,x)\subseteq\mathbb N$ that satisfies:
\begin{enumerate}
\item $\lim_{N\to\infty}\frac 1N\#\big(\{1,\cdots, N\}\cap\Lambda(M,H,x)\big)=1$.

\item For all $f\in\mathcal C_{M}(X)$ and any given $n\in \Lambda(M,H,x)$,  $ f(\hat{T}^{n+h}x)$ is constant for $0\leq h\leq H-1$.
\end{enumerate}

\end{Corollary}
\begin{proof} Let $\Lambda^{(\ell)}(M,H,x)$ be as in Lemma \ref{LemLongPeriod},  and set 
$$\Lambda(M,H,x):=\bigcap_{0\leq \ell\leq 3}\Lambda^{(\ell)}(M,H,x)\subset\mathbb N.$$
Then clearly we still have
$$\lim_{N\to\infty}\frac 1N\#\big(\{1,\cdots, N\}\cap\Lambda(M,H,x)\big)=1.$$ 
Next, let $f\in\mathcal C_{M} (X)$. Since $f(\hat{T}^{n+h}x)$ only depends on $\big([T^{n+h}x^{(\ell)}]_M\big)_{0\leq \ell\leq 3}$ and the $i$ coordinate (that does not change when we apply $\hat{T}$), given $n\in\Lambda^{(\ell)}(i,M, H, x)$, it is constant for $0\leq h\leq H-1$ by Lemma \ref{LemLongPeriod}.\end{proof}

We are now ready to establish M\"obius disjointness:

\begin{proof}[Proof of Theorem \ref{Main Theorem} Part (2)] As remarked in the beginning of this Section, we may assume $f\in\mathcal C_M(X)$ for some $M$ and $|f|\leq 1$. Let $x\in X$. Then for a fixed $H$, as $N\to\infty$,
\begin{equation*}
\begin{aligned}
\Big|\frac1N\sum_{n=1}^Nf(\hat{T}^nx)\mu(n)\Big|
=&\Big|\frac1N\sum_{n=1}^N\frac 1H\sum_{h=0}^{H-1}f(\hat{T}^{n+h}x)\mu(n)\Big|+O(\frac HN)\\
=&\Big|\frac1N\sum_{\substack{1\leq n\leq N\\n\in\Lambda(M,H,x)}}\frac1H \sum_{h=0}^{H-1}f(\hat{T}^{n+h}x)\mu(n)\Big|+o_H(1)+O(\frac HN)\\
\leq &\frac1N\sum_{\substack{1\leq n\leq N\\n\in\Lambda(M,H,x)}}\Big|\frac1H \sum_{h=0}^{H-1}f(\hat{T}^{n+h}x)\mu(n+h)\Big|+o_H(1)+O(\frac HN)
\end{aligned}
\end{equation*}
Here $o_H(1)$ stands for a quantity that tends to $0$ as $N\to\infty$ for a fixed $H$.

By Corollary \ref{CorLongPeriod}, $f(\hat{T}^{n+h}x)=f(\hat{T}^n x)$ for every $n\in\Lambda(M,H,x)$ and $0\leq h\leq H-1$. So,
\begin{equation*}
\begin{aligned}
\Big|\frac1N\sum_{n=1}^Nf(\hat{T}^nx)\mu(n)\Big|
\leq &  \frac1N\sum_{\substack{1\leq n\leq N\\n\in\Lambda(M,H,x)}}\Big|\frac1H \sum_{h=0}^{H-1}f(\hat{T}^nx)\mu(n+h)\Big|+o_H(1)+O(\frac HN)\\
\leq  & \frac1N\sum_{\substack{1\leq n\leq N\\n\in\Lambda(M,H,x)}}\Big|\frac1H \sum_{h=0}^{H-1}\mu(n+h)\Big|+o_H(1)+O(\frac HN)\\
\leq &  \frac1N\sum_{n=1}^N\Big|\frac1H \sum_{h=0}^{H-1}\mu(n+h)\Big|+o_H(1)+O(\frac HN)\\
= & O\Big((\frac1{\log H})^{0.01}+(\frac{\log H}{\log N})^{0.01}\Big)+o_H(1)+O(\frac HN).
\end{aligned}
\end{equation*}
 The last step is given by \cite[Theorem 1]{MR16}. 

By letting $H\to \infty$ first, and then $N\to\infty$ for each fixed $H$, we see that 
$$\frac1N\sum_{n=1}^Nf(\hat T^nx)\mu(n)=o(1) \text{ as } N\to\infty.$$
\end{proof}

\bibliography{bib}
\bibliographystyle{plain}

\end{document}